\documentclass{agtart_a}
\pdfoutput=1

\usepackage{amscd}

\input xy
\xyoption{all}

\title{The normaliser decomposition for $p$--local finite groups}

\author{Assaf Libman}
\givenname{Assaf}
\surname{Libman}
\address{Department of Mathematical Sciences\\
King's College\\
University of Aberdeen\\\newline
Aberdeen\\
AB24 3UE\\
Scotland\\
United Kingdom}
\email{assaf@maths.abdn.ac.uk}
\urladdr{}

\volumenumber{6}
\issuenumber{}
\publicationyear{2006}
\papernumber{44}
\startpage{1267}
\endpage{1288}

\doi{}
\MR{}
\Zbl{}


\received{4 February 2005}
\revised{}
\accepted{29 June 2006}
\published{}
\publishedonline{11 September 2006}
\proposed{}
\seconded{}
\corresponding{}
\editor{}
\version{}

\arxivreference{}

\let\xysavmatrix\xymatrix
\def\xymatrix{\disablesubscriptcorrection\xysavmatrix}
\AtBeginDocument{\let\bar\wbar\let\tilde\wtilde\let\hat\what}

\def\underk{\mskip0mu\underline{\mskip-0mu{k}\mskip-3mu}\mskip3mu}
\def\undern{\mskip0mu\underline{\mskip-0mu{n}\mskip-2mu}\mskip2mu}

\makeatletter
\def\cnewtheorem#1[#2]#3{\newtheorem{#1}{#3}[section]
\expandafter\let\csname c@#1\endcsname\c@thmain}
\def\dnewtheorem#1[#2]#3{\newtheorem{#1}{#3}[section]
\expandafter\let\csname c@#1\endcsname\c@prop}
\makeatother

\newtheorem{thma}{Theorem A}
\makeautorefname{thma}{Theorem A}

\newtheorem{thmb}{Theorem B}
\makeautorefname{thmb}{Theorem B}

\cnewtheorem{propmain}[thmain]{Proposition}

\swapnumbers
\newtheorem{prop}{Proposition}[section]
\dnewtheorem{thm}[prop]{Theorem}
\newtheorem{cofthm}{Cofinality Theorem}
\makeautorefname{cofthm}{Cofinality Theorem}

\newtheorem{pdthm}{Segal's Pushdown Theorem}
\makeautorefname{pdthm}{Segal's Pushdown Theorem}

\dnewtheorem{cor}[prop]{Corollary}
\dnewtheorem{lem}[prop]{Lemma}

\theoremstyle{definition}
\dnewtheorem{defn}[prop]{Definition}
\makeautorefname{defn}{Definition}
\dnewtheorem{void}[prop]{}
\makeautorefname{void}{}
\dnewtheorem{remark}[prop]{Remark}
\newtheorem{remarkx}{Remark}

\newtheorem{notationx}{Notation}

\makeop{Mor}
\makeop{Iso}
\makeop{Aut}
\makeop{Obj}
\makeop{Hom}

\def\vp{\varphi}

\def\de{\delta}

\def\A{\ensuremath{\mathcal{A}}}
\def\B{\ensuremath{\mathcal{B}}}
\def\D{\ensuremath{\mathcal{D}}}
\def\E{\ensuremath{\mathcal{E}}}
\def\F{\ensuremath{\mathcal{F}}}
\def\H{\ensuremath{\mathcal{H}}}
\def\I{\ensuremath{\mathcal{I}}}
\def\J{\ensuremath{\mathcal{J}}}

\def\L{\ensuremath{\mathcal{L}}}

\def\T{\ensuremath{\mathcal{T}}}

\def\NN{\mathbb{N}}
\def\ZZ{\mathbb{Z}}

\def\spaces{\textbf{Spaces}}

\def\incl{\text{incl}}
\def\id{\textrm{id}}
\def\Id{\textrm{Id}}
\def\op{\ensuremath{\mathrm{op}}}

\makeop{res}
\makeop{hocolim}
\makeop{sk}
\makeop{sd}

\def\bsd{{\mathrm{\bar{s}d}}}

\def\Cat{\mathbf{Cat}}
\def\cat{\Cat}
\def\sets{\mathbf{Sets}}

\def\Nr{\text{Nr}}
\def\Tr{\text{Tr}}

\def\No{\breve{N}}
 
\def\AAA{{\mathbf{A}}}
\def\BBB{{\mathbf{B}}}
\def\EEE{{\mathbf{E}}}
\def\HHH{{\mathbf{H}}}
\def\KKK{{\mathbf{K}}}
\def\LLL{{\mathbf{L}}}
\def\PPP{{\mathbf{P}}}

\def\bfK{{\mathbf{K}}}

\def\darrow{\downarrow}

\newcommand{\xto}[1]{\xrightarrow{#1}}
\newcommand{\func}[3]{\ensuremath{#1\co #2\rightarrow#3}}

\newcommand{\pcomp}[1]{{#1}^\wedge_p}

\newcommand{\hhocolim}[1]{\hocolim_{#1}}

\numberwithin{equation}{section}

\begin{document}

\begin{asciiabstract}
We construct an analogue of the normaliser decomposition for p-local
finite groups (S,F,L) with respect to collections of F-centric subgroups
and collections of elementary abelian subgroups of S.  This enables
us to describe the classifying space of a p-local finite group, before
p-completion, as the homotopy colimit of a diagram of classifying spaces
of finite groups whose shape is a poset and all maps are induced by
group monomorphisms.
\end{asciiabstract}

\begin{htmlabstract}
We construct an analogue of the normaliser decomposition for p&ndash;local
finite groups (S,F,L) with respect to collections
of F&ndash;centric subgroups and collections of elementary abelian
subgroups of S.  This enables us to describe the classifying space
of a p&ndash;local finite group, before p&ndash;completion, as the homotopy
colimit of a diagram of classifying spaces of finite groups whose shape
is a poset and all maps are induced by group monomorphisms.
\end{htmlabstract}

\begin{abstract}
We construct an analogue of the normaliser decomposition for $p$--local
finite groups $(S,\mathcal{F},\mathcal{L})$ with respect to collections
of $\mathcal{F}$--centric subgroups and collections of elementary abelian
subgroups of $S$.  This enables us to describe the classifying space
of a $p$--local finite group, before $p$--completion, as the homotopy
colimit of a diagram of classifying spaces of finite groups whose shape
is a poset and all maps are induced by group monomorphisms.
\end{abstract}

\maketitle

%
%
\section{The main results}

For finite groups Dwyer \cite{Dwyer-decomp} defined three types of
homology decompositions of classifying spaces of finite groups known as the 
``subgroup'', ``centraliser'' and ``normaliser'' decompositions.
These decompositions are functors $F\co D \to \spaces$, where
$D$ is a small category which is constructed using collections $\H$ of carefully chosen
subgroups of $G$.
The essential property of these functors is, that given a finite group
$G$, the spaces $F(d)$ have the homotopy type of classifying spaces of
subgroups of $G$.
Moreover the category $D$ is constructed using information about the conjugation
in $G$ of the subgroups in $\H$.
We say that $D$ depends on the fusion of the collection $\H$ of $G$.

The purpose of this note is to construct an analogue of the normaliser
decomposition for $p$--local finite groups in certain important cases.
Throughout this note we will freely use the terminology and notation that 
by now has become standard in the theory for $p$--local finite groups.
The reader who is not familiar with the jargon is advised to read \fullref{sec
  plfg} prior to this section, and is also referred to \cite{BLO2}
where $p$--local finite groups were initially defined. 

It should be noted that
the analogues of the ``subgroup'' and the
``centraliser'' decompositions for $p$--local finite groups was already
known to Broto, Levi and Oliver \cite[Section~2]{BLO2}.

The normaliser decomposition which is introduced in this note 
enabled the author together with Antonio Viruel to analyze the nerve $|\L|$
of $p$--local finite groups $(S,\F,\L)$ with small Sylow
subgroups $S$.
We prove that these are classifying spaces of, generally infinite,
discrete groups \cite{LV}. 
The author also used normaliser decompositions to give an analysis of the spectra associated
with the nerve, $|\L|$, of the linking systems due to 
Ruiz and Viruel in \cite{RV-extraspecial} and other ``exotic'' examples, see \cite{Li}.
These results will appear separately as they involve techniques that
have little to do with the actual construction of the normaliser
decomposition.  

We now describe the main results of this paper.
Throughout we work simplicially, thus a space means a
simplicial set.
The category of simplicial sets is denoted by $\spaces$.
The nerve of a small category $\mathbf{D}$ is denoted
$\Nr(\mathbf{D})$ or $|\mathbf{D}|$.
We obtain a functor $|-|\co\cat\to\spaces$ where $\cat$ is the category
of small categories. 
A more detailed discussion can be found in \fullref{sec homotopy colimits}

\begin{defn}
Let $(S,\F,\L)$ be a $p$--local finite group.
A collection is a set ${\mathcal{C}}$ of subgroups of $S$ which is
closed under conjugacy in $\F$.
That is if $P\leq S$ belongs to ${\mathcal{C}}$ then so do all the $\F$--conjugates of $P$.
A collection ${\mathcal{C}}$ is called $\F$--centric if it consists of
$\F$--centric subgroups of $S$.
\end{defn}

\begin{defn}
\label{def k-simplices}
A $k$--simplex in a collection ${\mathcal{C}}$ is a sequence $\PPP$ of proper inclusions 
$P_0<P_1<\cdots <P_k$ of elements of ${\mathcal{C}}$.
Two $k$--simplices $\PPP$ and $\PPP'$ are called conjugate if there exists an isomorphism
$f\in \Iso_\F(P_k,P_k')$ such that $f(P_i)=P_i'$ for all $i=0,\ldots,k$.
The conjugacy class of $\PPP$ is denoted $[\PPP]$. 
\end{defn}

\begin{defn}
\label{def bsdc}
The category $\bsd{\mathcal{C}}$ is a poset whose objects are the conjugacy classes $[\PPP]$ of all
the $k$--simplices in ${\mathcal{C}}$ where $k=0,1,2,\ldots$.
A morphism $[\PPP] \to [\PPP']$ in $\bsd{\mathcal{C}}$ exists if $\PPP'$ is conjugate to a
subsimplex of $\PPP$.
\end{defn}
Recall from \fullref{iota morphisms} that in every $p$--local finite group
it is possible to choose morphisms $\iota_P^Q$ in the linking system $\L$ which are
lifts of inclusions $P\leq Q$ of $\F$--centric subgroups.
The choice can be made in such a way that
$\iota_Q^R\circ\iota_P^Q=\iota_P^R$ for inclusions $P\leq Q\leq R$.

\begin{defn}
Let ${\mathcal{C}}$ be an $\F$--centric collection in $(S,\F,\L)$ and let
$\PPP$ be a $k$--simplex in ${\mathcal{C}}$.
Define $\Aut_\L(\PPP)$ as the subgroup of $\prod_{i=0}^k\Aut_\L(P_i)$ whose elements
are the $(k{+}1)$--tuples $(\vp_i)_{i=0}^k$ which render the following ladder commutative
in $\L$
$$
\begin{CD}
P_0 
@>{\iota_{P_0}^{P_1}}>>
P_1
@>{\iota_{P_1}^{P_2}}>>
\cdots
@>{\iota_{P_{k-1}}^{P_k}}>>
P_k 
\\
@V{\vp_0}VV
@V{\vp_1}VV
@.
@VV{\vp_k}V
\\
P_0 
@>>{\iota_{P_0}^{P_1}}>
P_1
@>>{\iota_{P_1}^{P_2}}>
\cdots
@>>{\iota_{P_{k-1}}^{P_k}}>
P_k 
\end{CD}
$$
\end{defn}

\begin{prop}
\label{prop resL}
The assignment $(\vp_i)_{i=0}^k \mapsto \vp_0$ gives rise to a canonical isomorphism of
$\Aut_\L(\PPP)$ with a subgroup of $\Aut_\L(P_0)$.
More generally, if $\PPP'$ is a subsimplex of $\PPP$ in ${\mathcal{C}}$ then
restriction 
induces a monomorphism of groups
$\Aut_\L(\PPP) \to \Aut_\L(\PPP')$.
\end{prop}

\begin{proof}
The second assertion follows immediately from \fullref{restn in L}.
The first follows from the second by letting $\PPP'$ be the $1$--simplex $P_0$.
\end{proof}

\begin{notationx}
$\B\Aut_\L(\PPP)$ denotes the subcategory of $\L$ whose only object is $P_0$
and whose morphism set is $\Aut_\L(\PPP)$.
\end{notationx}

\begin{defn}
Given an $\F$--centric collection ${\mathcal{C}}$ in a $p$--local finite group $(S,\F,\L)$, let
$\L^{\mathcal{C}}$ denote the full subcategory of $\L$ generated by the objects set ${\mathcal{C}}$.
\end{defn}

Frequently, the inclusion $\L^{\mathcal{C}} \subseteq \L$ induces a weak homotopy equivalence on nerves.
For example, this happens when ${\mathcal{C}}$ contains all the
$\F$--centric $\F$--radical subgroups of $S$.
This fact is proved by Broto, Castellana, Grodal, Levi and Oliver
\cite[Theorem 3.5]{BCGLO1}.

The following theorem applies to all $\F$--centric collections.
The decomposition approximates $\L$ if the inclusion $\L^{\mathcal{C}}
\subseteq \L$ induces an equivalence as explained above.

\begin{thma}
\label{thmA}
Fix an $\F$--centric collection ${\mathcal{C}}$ in a $p$--local finite group $(S,\F,\L)$.
Then there exists a functor $\de_{\mathcal{C}}\co\bsd{\mathcal{C}}\to\spaces$ such that
\begin{enumerate}
\item
\label{thmA:target}
There is a natural weak homotopy equivalence 
$$\hhocolim{\bsd{\mathcal{C}}}\, \de_{\mathcal{C}} \xto{~~\simeq~~}
  |\L^{\mathcal{C}}|.$$
\item There is a natural weak  homotopy equivalence
$B\Aut_\L(\PPP) \xto{\simeq} \de_{\mathcal{C}}([\PPP])$ 
for every $k$--simplex $\PPP$.
\label{thmA:terms}

\item
The natural maps $\de_{\mathcal{C}}([\PPP]) \to |\L^{\mathcal{C}}|$ are induced by the inclusion of 
categories $\B\Aut_\L(\PPP) \subseteq \L^{\mathcal{C}}$.
\label{thmA:augmentation}

\item
If $\PPP'$ is a subsimplex of $\PPP$ then the equivalence \eqref{thmA:terms} renders 
the following square commutative
$$
\xymatrix{
B\Aut_\L(\PPP) \ar[r]^\simeq \ar[d]_{B\res^\PPP_{\PPP'}} &
\de_{\mathcal{C}}([\PPP]) \ar[d] 
\\
B\Aut_\L(\PPP') \ar[r]^\simeq  &
\de_{\mathcal{C}}([\PPP']) 
}
$$
Moreover if $\PPP$ and $\PPP'$ are conjugate $k$--simplices and $\psi\in\Iso_\L(P_0,P_0')$
maps $\Aut_\L(\PPP')$ onto $\Aut_\L(\PPP)$ by conjugation then the following square
commutes
$$
\xymatrix{
B\Aut_\L(\PPP') \ar[r]^{\simeq} \ar[d]_{Bc_\psi} &
\de_{\mathcal{C}}([\PPP']) \ar@{=}[d] 
\\
B\Aut_\L(\PPP) \ar[r]_{\simeq} &
\de_{\mathcal{C}}([\PPP]) 
}
$$
\label{thmA:maps}
\end{enumerate}
\end{thma}

\begin{remarkx}
When $(S,\F,\L)$ is associated with a finite group $G$ one may consider
the $G$--collection
$\H$ consisting of all the subgroups of $G$ which are conjugate to elements of the 
$\F$--collection ${\mathcal{C}}$.
Dwyer \cite[Section~3]{Dwyer-decomp} constructs a poset $\bsd\H$ and a functor 
$\de^{\text{Dwyer}}_\H\co\bsd\H \to \spaces$
which he calls the normaliser decomposition.
We will show in \fullref{compare with Dwyer} that $\bsd\H=\bsd{\mathcal{C}}$ and that $\de_{\mathcal{C}}$ and 
$\de^{\text{Dwyer}}_\H$ can be connected by a natural zigzag of mod--$p$ equivalences.
That is, a zigzag of natural transformations which
at every object of $\bsd{\mathcal{C}}$ give rise to an $H_*(-;\ZZ/p)$--isomorphism.
\end{remarkx}
We now describe the second type of normaliser decomposition that we shall 
construct in this note.
It is based on collections $\E$ of elementary abelian subgroups of $S$.

\begin{defn}
\label{def autF}
For a $k$--simplex $\EEE$ in $\E$ define $\Aut_\F(\EEE)$ as the subgroup of 
$\Aut_\F(E_k)$ consisting of the automorphisms $f$ such that $f(E_i)=E_i$ 
for all $i=0,\ldots,k$.
\end{defn}

Consider an $\F$--centric collection ${\mathcal{C}}$ in $(S,\F,\L)$.

\begin{defn}
\label{def barcl}
Fix an elementary abelian subgroup $E$ of $S$.
The objects of the category $\bar{C}_{\L}({\mathcal{C}};E)$ are pairs $(P,f)$ where $P\in {\mathcal{C}}$ and
$f\co E \to Z(P)$ is a morphism in $\F$.
Morphisms $(P,f) \to (Q,g)$ in $\bar{C}_{\L}({\mathcal{C}};E)$ are morphisms $\psi \in \L(P,Q)$ such that 
$g=\pi(\psi)\circ f$ where $\pi\co\L \to \F$ is the projection functor.
\end{defn}

Observe that $\Aut_\F(E)$ acts on $\bar{C}_{\L^{\mathcal{C}}}(E)$ by pre-composition.
That is, every $h\in\Aut_\F(E)$ indices the assignment $(P,f)\mapsto (P,f\circ h)$.

\begin{defn}
\label{def NoL}
For a $k$--simplex $\EEE$ in $\E$ let $\No_\L({\mathcal{C}};\EEE)$ denote the subcategory of $\L$
whose objects are $P\in{\mathcal{C}}$ for which $E_k \leq Z(P)$.
A morphism $\vp\in\L(P,Q)$ belongs to $\No_\L({\mathcal{C}};\EEE)$ if $\pi(\vp)|_{E_k}$ is an element
of $\Aut_\F(\EEE)$.
\end{defn}
Recall that the homotopy orbit space of a $G$--space $X$, ie the Borel construction
$EG\times_G X$, is denoted by $X_{hG}$.

\begin{prop}
\label{No equiv orbit}
Let $\EEE$ be a $k$--simplex in $\E$.
There is a map
$$
\epsilon_\EEE \co | \No_\L({\mathcal{C}};\EEE)| \to |\bar{C}_\L({\mathcal{C}};E_k)|_{h\Aut_\F(\EEE)}
$$ 
which is a homotopy equivalence if $E_k$ is fully $\F$--centralised.
The map is natural with respect to inclusion of simplices.
\end{prop}

\begin{proof}
This is immediate from \fullref{NoTrCbar}.
\end{proof}

A comment on the categories $\bar{C}_{\L}({\mathcal{C}};E)$ is in place.
If ${\mathcal{C}}$ is the collection of all the $\F$--centric subgroups of
$S$ and $E$ is fully $\F$--centralised, then it is shown by Broto, Levi
and Oliver \cite[Theorem 2.6]{BLO2} that $|\bar{C}_\L({\mathcal{C}};E)|$
has the homotopy type of the nerve of the centraliser linking system
$|C_\L(E)|$.  The categories $\No_\L({\mathcal{C}};\EEE)$ are more
mysterious.  Even when $\EEE$ is a $1$--simplex $E$, the category
$\No_\L({\mathcal{C}};E)$ is in general only a subcategory of the
normaliser linking system $N_\L(E)$ because the largest subgroup which
appears as an object of $\No_\L({\mathcal{C}};E)$ is $C_S(E)$ which in
general is  smaller than $N_S(E)$.  When $C_S(E)=N_S(E)$ these categories
are equal.

The next decomposition result, \fullref{thmB}, depends on a collection of 
elementary abelian groups
$\E$ and a collection ${\mathcal{C}}$ of $\F$--centric subgroups of $S$.
It approximates $\L$ if ${\mathcal{C}}$ contains, for example, all the
$\F$--centric $\F$--radical
subgroups of $S$.
The collection $\E$ must be large enough as explicitly stated in the theorem.
For example the collection of all the non-trivial elementary abelian subgroups will always be
a valid choice.

\begin{defn}
\label{def omega_p}
Given a group $H$ and a prime $p$ let $\Omega_p(H)$ denote the subgroup of $H$ generated
by all the elements of order $p$ in $H$.
\end{defn}

\begin{thmb}
\label{thmB}
Consider a $p$--local finite group $(S,\F,\L)$, an $\F$--centric collection ${\mathcal{C}}$ and a collection
$\E$ of elementary abelian subgroup of $S$ which contains the subgroups $\Omega_pZ(P)$
for all $P\in {\mathcal{C}}$.
Then there exists a functor $\de_\E\co\bsd\E \to \spaces$ with the following properties.
\begin{enumerate}
\item
There is a natural weak homotopy equivalence
$
\hhocolim{\bsd\E} \, \de_\E \xto{\ \ \simeq \ \ } |\L^{\mathcal{C}}|.
$
\label{thmB:target}

\item
For a $k$--simplex $\EEE$ in $\E$ there is a
weak homotopy equivalence
\label{thmB:terms}
$$
|\bar{C}_{\L}({\mathcal{C}};E_k)|_{h\Aut_\F(\EEE)} 
\xto{ \ \ \simeq \ \ }
\de_\E([\EEE]).
$$
\item
Fix a $k$--simplex $\EEE$ where $E_k$ is fully $\F$--centralised.
The equivalences \eqref{thmB:target} and \eqref{thmB:terms} give a natural map 
$\de_\E([\EEE]) \to |\L^{\mathcal{C}}|$ whose precomposition with $\epsilon_\EEE$ 
of \fullref{No equiv orbit} is induced by the realization of  the inclusion
of $\No_\L({\mathcal{C}};\EEE)$ in $\L^{\mathcal{C}}$.
\label{thmB:augment}

\item
If $\EEE'$ is a $k$--subsimplex of an $n$--simplex $\EEE$  then the following square commutes 
up to homotopy
$$
\xymatrix{
|\bar{C}_{\L}({\mathcal{C}};E_n)|_{h\Aut_\F(\EEE)} 
\ar[d] \ar[rr]^\simeq & &
\de_\E([\EEE]) \ar[d] 
\\
|\bar{C}_{\L}({\mathcal{C}};E'_k)|_{h\Aut_\F(\EEE')} 
\ar[rr]^\simeq & &
\de_\E([\EEE']) 
}
$$
The homotopy is natural with respect to inclusion of simplices.
In addition, the square commutes on the nose if $E_k'=E_n$.
\label{thmB:maps}
\end{enumerate}
\end{thmb}

\subsubsection*{Acknowledgments}
The author was supported by grant NAL/00735/G from the Nuf\-field Foundation.
Part of this work was supported by Institute Mittag-Leffler (Djursholm,
Sweden).

%
%
\section{On $p$--local finite groups}
\label{sec plfg}

The term $p$--local finite group was coined by Broto, Levi and Oliver \cite{BLO2}.
It cropped up naturally in their attempt \cite{BLO1} to
describe the space of self equivalences of a $p$--completed classifying space of a finite 
group $G$.
They discovered that the relevant information needed to solve this problem lies in the
fusion system of the $p$--subgroups of $G$ and certain categories which they later on called 
``linking systems''.
Historically, fusion systems were first introduced by Lluis Puig \cite{Puig}.

\begin{defn}
Fix a prime $p$ and let $S$ be a finite $p$--group.
A \emph{fusion system} over $S$ is a sub-category $\F$ of the category
of groups whose objects are the subgroups of $S$ and whose morphisms
are group monomorphisms such that
\begin{itemize}
\item[(1)]
All the monomorphisms that are induced by conjugation by elements of
$S$ are in $\F$.

\item[(2)]
Every morphism in $\F$ factors as an isomorphism in $\F$ followed by
an inclusion of subgroups.
\end{itemize}
We say that two subgroups $P,Q$ of $S$ are $\F$--\emph{conjugate} if
they are isomorphic as objects of $\F$.
\end{defn}

When $g$ is an element of $S$ and $P,Q$ are subgroups of $S$ such that
$g P g^{-1}\leq Q$, we let $c_g$ denote the morphism $P \to Q$ defined
by conjugation, namely $c_g(x)=gxg^{-1}$ for every $x\in P$.

We let $\Hom_S(P,Q)$ denote the set of all the morphisms $P \to Q$ in
$\F$ that are induced by conjugation in $S$.
Also notice that the factorization axiom (2) implies that all the
$\F$--endomorphisms of a subgroup $P$ are in fact automorphisms in $\F$.
Thus we write $\Aut_\F(P)$ for the set of morphisms $\F(P,P)$.

\begin{defn}
A subgroup $P$ of $S$ is called \emph{fully $\F$--centralised} 
(resp. \emph{fully $\F$--normalised}) if
its $S$--centraliser $C_S(P)$ (resp. $S$--normaliser $N_S(P)$) has the
maximal possible order in the
$\F$--conjugacy class of $P$.
That is, $|C_S(P)|\geq |C_S(P')|$ (resp. $|N_S(P)|\geq |N_S(P')|$) for
every $P'$ which is 
$\F$--conjugate to $P$.
\end{defn}

\begin{defn}
\label{def sat fus}
A fusion system $\F$ over a finite $p$--group $S$ is called
\emph{saturated} if
\begin{itemize}
\item[I]
Every fully $\F$--normalised subgroup $P$ of $S$ is fully $\F$--centralised and
moreover $\Aut_S(P)=N_S(P)/C_S(P)$ is a Sylow $p$--subgroup of $\Aut_\F(P)$.

\item[II]
Every morphism $\func{\vp}{P}{S}$ in $\F$ whose image $\vp(P)$ is fully
$\F$--centralised can be extended to a morphism $\func{\psi}{N_\vp}{S}$ in $\F$
where
$$
N_{\vp}=\{ g\in N_S(P) : \vp c_g\vp^{-1} \in\Aut_S(P)\}.
$$
\end{itemize}
\end{defn}

\begin{defn}
A subgroup $P$ of $S$ is called $\F$--\emph{centric} if $P$ and all of
its $\F$--conjugates contain their $S$--centralisers, that is
$C_S(P')=Z(P')$ for every subgroup $P'$ of $S$ which is $\F$--conjugate
to $P$.
\end{defn}

\begin{defn}
A \emph{centric linking system} associated to a saturated fusion system $\F$ over $S$
consists of
\begin{enumerate}
\item
A small category $\L$ whose objects are the $\F$--centric subgroups
of $S$,

\item
a functor $\pi\co\L \to \F$ and

\item
group monomorphisms $\de_P\co P \to \Aut_\L(P)$ for every $\F$--centric
subgroup $P$ of $S$,
\end{enumerate}
Such that the following axioms hold
\begin{itemize}
\item[(A)]
The functor $\pi$ acts as the inclusion on object sets, that is $\pi(P)=P$
for every $\F$--centric subgroup $P$ of $S$.
For any two objects $P,Q$ of $\L$, the group $Z(P)$ acts freely on the
morphism set $\L(P,Q)$ via the restriction of $\de_P\co P \to \Aut_\L(P)$
to $Z(P)$.
The induced map on morphisms sets
$$
\pi\co\L(P,Q) \to \F(P,Q)
$$
identifies $\F(P,Q)$ with the quotient of $\L(P,Q)$ by the free action of $Z(P)$.

\item[(B)]
For every $\F$--centric subgroup $P$ of $S$ the map $\pi\co\Aut_\L(P) \to
\Aut_\F(P)$ sends $\de_P(g)$, where $g\in P$, to $c_g$.

\item[(C)]
For every $f\in\L(P,Q)$ and every $g\in P$ there is a commutative
square in $\L$
$$
\begin{CD}
P @>{f}>> Q
\\
@V{\de_P(g)}VV
@VV{\de_Q(\pi(f)(g))}V
\\
P @>>{f}> Q.
\end{CD}
$$
\end{itemize}
\end{defn}

\begin{remarkx}
A morphism $f\in\L(P,Q)$ is called a \emph{lift} of a morphism
$\vp\in\F(P,Q)$ if $\vp=\pi(f)$. 
\end{remarkx}

\begin{defn}
A $p$--\emph{local finite group} is  a triple $(S,\F,\L)$ where $\F$ is
a saturated fusion system over the finite $p$--group $S$ and $\L$ is a
centric linking system associated to $\F$.
The \emph{classifying space} of $(S,\F,\L)$ is the space
$\pcomp{|\L|}$, that is the $p$--completion in the sense of Bousfield
and Kan \cite{BK}, of the realization of the small category $\L$.
\end{defn}

\begin{void}
\label{non exotic examples}
When $S$ is a Sylow $p$--subgroup of a finite group $G$, there is an
associated $p$--local finite group denoted $(S,\F_S(G),\L_S(G))$.
See \cite[Proposition~1.3, remarks after Definition~1.8]{BLO2}.
We shall write $\F$ for $\F_S(G)$ and $\L$ for $\L_S(G)$.

Morphism sets between $P,Q\leq S$ are
$$
\F(P,Q)=\Hom_G(P,Q)=N_G(P,Q)/C_G(P)
$$
where $N_G(P,Q)=\{g\in G : gPg^{-1}\leq Q\}$ and $C_G(P)$ acts on $N_G(P,Q)$ 
by right translation. 

A subgroup $P$ of $S$ is, by \cite[Proposition~1.3]{BLO2}, $\F$--centric precisely when it is
$p$--centric in the sense of \cite[Section~1.19]{Dwyer-decomp},
that is, $Z(P)$ is a Sylow $p$--subgroup of $C_G(P)$.
In this case $C_G(P)=Z(P)\times C'_G(P)$ where $C_G'(P)$ is the
maximal subgroup of $C_G(P)$ of order prime to $p$.
Morphism sets of $\L=\L_S(G)$ have, by definition, the form
$$
\L(P,Q) = N(P,Q)/C_G'(P).
$$
The functor $\pi\co\L_S(G)\to\F_S(G)$ is the obvious projection functor.
The monomorphism $\de_P\co P\to \Aut_\L(P)$ is induced by the inclusion of $P$ in $N_G(P)$.

It is shown by Broto, Levi and Oliver \cite[after Definition~1.8]{BLO2}
that $(S,\F_S(G),\L_S(G))$ is a
$p$--local finite group and that $\pcomp{|\L_S(G)|}\simeq \pcomp{BG}$.
It should also be remarked that there are examples of $p$--local finite groups that cannot be
 associated with any finite group.
These are usually referred to as ``exotic examples''.
\end{void}

\begin{void}
\label{iota morphisms}
In every $p$--local finite group $(S,\F,\L)$ one can choose morphisms  $\iota_P^Q\in\L(P,Q)$ for
every inclusion of 
$\F$--centric subgroups $P\leq Q$, in such a way that
\begin{enumerate}
\item
$\pi(\iota_P^Q)$ is the inclusion $P\leq Q$,
\item
$\iota_Q^R\circ \iota_P^Q=\iota_P^R$ for every $\F$--centric subgroups $P\leq
  Q\leq R$ of $S$, and
\item
$\iota_P^P=\id$ for every $\F$--centric subgroup $P$ of $S$.
\end{enumerate}
This follows from \cite[Proposition 1.11]{BLO2}.
Using the
notation there, one chooses $\iota^Q_P=\delta_{P,Q}(e)$ where $e$ is
the identity element in $S$.
Whenever possible, in order to avoid cumbersome notation, we shall
write $\iota$ for $\iota_P^Q$.
\end{void}

\begin{void}
\label{unique factor in L}
From \cite[Lemma 1.10(a)]{BLO2} it also follows that every morphism
$\vp\co P \to Q$ in $\L$ factors uniquely as an isomorphism $\vp'\co P \to
P'$ in $\L$ followed by the morphism $\iota\co P' \to Q$.
In fact $P'=\pi(\vp)(P)$
\end{void}

\begin{void}
\label{morphisms in L are mono and epi}
It was observed by Broto, Levi and Oliver \cite[remarks after
Lemma~1.10]{BLO2} that
every morphism in $\L$ 
is a monomorphism in the categorical sense.
It was later observed by Broto, Castellana, Grodal, Levi and Oliver
\cite[Corollary 3.10]{BCGLO1} and independently by others, that every
morphism in $\L$ is also an epimorphism.  As an easy consequence we
record for further use:
\end{void}

\begin{prop}
\label{restn in L}
Consider a $p$--local finite group $(S,\F,\L)$ and a commutative square in $\F$ on the
left of the display below
$$
\begin{CD}
P @>{f}>> P'
\\
@V{\incl}VV @VV{\incl}V
\\
Q @>>{g}> Q'
\end{CD}
\qquad
\qquad
\qquad
\qquad
\qquad
\begin{CD}
P @>{\tilde{f}}>> P'
\\
@V{\iota_P^Q}VV @VV{\iota_{P'}^{Q'}}V
\\
Q @>>{\tilde{g}}> Q'
\end{CD}
$$ 
where $P,P',Q$ and $Q'$ are $\F$--centric subgroups of $S$.
Then for every lift $\tilde{g}$ of $g$ in $\L$ there exists a unique lift $\tilde{f}$
of $f$ in $\L$ which render the square on the right commutative in $\L$.
We denote $\tilde{f}$ by $\tilde{g}|_P$.

Given a lift $\tilde{f}$ for $f$, if there exists a lift $\tilde{g}$ for $g$ rendering the 
square on the right commutative, then it is unique.
\end{prop}

\begin{proof}
The first assertion follows immediately from \cite[Lemma 1.10(a)]{BLO2} by setting
$\psi=\incl_{P'}^{Q'}, \tilde{\psi}=\iota_{P'}^{Q'}$ and $\tilde{\psi\vp}=\tilde{g}\iota_P^Q$.
The second assertion follows immediately from the fact that $\iota_P^Q$ is an epimorphism.
\end{proof}

\begin{void}
\label{centraliser and normaliser systems}
Fix a $p$--local finite group $(S,\F,\L)$.
Given a subgroup $P$ of $S$, there are two important $p$--local finite
groups associated with it: the centraliser of $P$ when $P$ is fully $\F$--centralised and
the normaliser of $P$ when $P$ is fully $\F$--normalised.
Both were defined by Broto Levi and Oliver in \cite{BLO2}.

The centraliser fusion system $C_\F(P)$, where $P$ is fully
$\F$--centralised, is a subcategory of $\F$.
As a fusion system it is defined over the $S$--centraliser of $P$
denoted $C_S(P)$. 
Morphisms $Q \to Q'$ in $C_\F(P)$ are those morphisms $\vp\co Q \to Q'$ in
$\F$ that can be extended to a morphism $\bar{\vp}\co PQ \to PQ'$ in $\F$
which induces the identity on $P$.
The objects of the centric linking system $C_\L(P)$ associated to
$C_\F(P)$ are
the $C_\F(P)$--centric subgroups of $C_S(P)$.
The set of morphisms $Q \to Q'$ in $C_\L(P)$ is a subset of $\L(PQ,PQ')$
consists of those morphisms $f\co PQ \to PQ'$ such that $\pi(f)$ induces
the identity on $P$ and carries $Q$ to $Q'$.
It is shown in \cite{BLO2} that $(C_S(P),C_\F(P),C_\L(P))$ is a
$p$--local finite group.

Now fix a subgroup $K \leq \Aut_\F(P)$ where $P$ is fully normalised
in $\F$.
The $K$--normaliser fusion system $N^K_\F(P)$ is a subcategory of $\F$
defined over $N_S(P)$. 
The objects of $N_\F^K(P)$ are the subgroups of $N_S(P)$.
A morphisms $\vp\in\F(Q,Q')$ belongs to $N^K_\F(P)$ if it
can be extended to a morphism $\bar{\vp}\co PQ \to PQ'$ in $\F$ which
induces an automorphism from $K$ on $P$.
The fusion system $N_\F^K(P)$ is saturated.
When $K=\Aut_\F(P)$ we denote this category by $N_\F(P)$ and call it
the normaliser fusion system of $P$.
The centric linking system $N_\L(P)$ associated to $N_\F(P)$ has the
$N_\F(P)$--centric subgroups of $N_S(P)$ as its object set.
The set of morphisms $Q \to Q'$ is the subset of $\L(PQ,PQ')$
consisting of those $f\co PQ \to PQ'$ such that $\pi(f)$ carries $Q$ to
$Q'$ and induces an automorphism on $P$.
\end{void}

%
%
\section{The Grothendieck construction}
\label{sec homotopy colimits}
Throughout this paper we work simplicially, namely a ``space'' means a
simplicial set.  For further details, the reader is referred to Bousfeld
and Kan \cite{BK}, May \cite{May-ss}, Goerss and Jardine \cite{Goerss-sht}
and many other sources.  In this section we collect several results from
general simplicial homotopy theory that we shall use repeatedly in the
rest of this note.


\textbf{Homotopy colimits}\qua 
Fix a small category $\KKK$ and a functor $U\co\KKK\to\spaces$.
The simplicial replacement of $U$ is the simplicial space $\coprod_*U$
which has in simplicial dimension $n$ the disjoint union of the spaces $U(K_0)$ for
every chain
$$K_0 \to K_1 \to \cdots \to K_n$$
of $n$ composable arrows in $\KKK$.
The homotopy colimit of $U$ denoted $\hhocolim{\KKK}U$ is the diagonal of $\coprod_*U$
regarded as a bisimplicial set. 
See Bousfield and Kan \cite[Section~XII.5]{BK}.

Consider a functor $F\co\KKK \to \LLL$ between small categories.
For a functor $U\co\LLL\to \spaces$ there is an obvious natural map, cf
\cite[Section~XI.9]{BK}.
$$
\hhocolim{\KKK} F^*U \to \hhocolim{\LLL} U.
$$
For an object $L\in \LLL$, the comma category $(L \darrow F)$ has
the pairs $\smash{(K, L \xto{k\in\KKK} FK)}$ as its objects.
Morphisms $(K,L
\smash{\stackrel{\raisebox{-1pt}{\scriptsize$k$}}{\longrightarrow}} FK)
\to (K',L
\smash{\stackrel{\raisebox{-1pt}{\scriptsize$k'$}}{\longrightarrow}} FK')$ are the morphisms
$x\co K \to K'$
such that $Fx \circ k = k'$.
Similarly one defines the category $(F \darrow L)$ whose object set consists of the pairs
$(K,k\co FK\to L)$.
Compare MacLane \cite{MacLane-working}.

\begin{defn}
\label{def right cofinal}
The functor $F\co\KKK \to \LLL$ is called right-cofinal if for every
object $L\in \LLL$ the category $(L \darrow F)$ has a contractible nerve.
\end{defn}

The following theorem was probably first proved by Quillen \cite[Theorem A]{Quillen}.
See also Hollender and Vogt \cite[Section~4.4]{Ho-Vo} and 
Bousfield and Kan \cite[Section~XI.9]{BK}.

\begin{cofthm}
\label{cofinal thm}
Let $F\co\KKK \to \LLL$ be a right cofinal functor between small categories.
Then for every functor $U\co\LLL\to\spaces$ the natural map 
$$\hhocolim{\KKK}F^*U\to\hhocolim{\LLL}U$$
is a weak homotopy equivalence.
\end{cofthm}
Associated with a functor $U\co\KKK\to\spaces$ there is a functor 
$F_*U\co\LLL\to\spaces$ called the homotopy left Kan extension of $U$ along $F$.
It is defined on every object $L\in\LLL$ by
$$
F_*U(L) = \hocolim \Big(
(F\darrow L) \xto{\text{proj}} \KKK \xto{F} \spaces
\Big).
$$
See \cite[Section~5]{Ho-Vo}, \cite[Section~6]{Dw-Ka}.
The following theorem is originally due to Segal.
See eg \cite[Theorem 5.5]{Ho-Vo}.

\begin{pdthm}
\label{pushdown thm}
Fix a functor $F\co\KKK\to\LLL$ of small categories.
Then for every functor $U\co\KKK\to\spaces$ there is a natural weak homotopy equivalence
$$
\hhocolim{\LLL} F_*U \xto{ \ \ \simeq \ \ } \hhocolim{\KKK}U.
$$
\end{pdthm}

\textbf{The Grothendieck construction}\qua
Recall that a small category $\KKK$ gives rise to a simplicial set
$\Nr(\KKK)$ called the \emph{nerve} of $\KKK$.
Its $n$--simplices are the chains of $n$ composable arrows 
$K_0\to K_1\to\cdots\to K_n$ in $\KKK$.
See, for example, Goerss and Jardine \cite[Example 1.4]{Goerss-sht} or
Bousfield and Kan \cite[Section~XI.2]{BK}.
We shall also use the notation $|\KKK|$ for the nerve of $\KKK$.

Given a functor $U\co\KKK\to\Cat$ Thomason \cite{Thomason} defined the
translation category $\KKK\int U$ associated to $U$ as follows.
The object set consists of pairs $(K,X)$ where $K$
is an object of $\KKK$ and $X$ is an object of $U(K)$.
Morphisms $(K_0,X_0)\to(K_1,X_1)$ are pairs $(k,x)$
where $k\co K_0\to K_1$ is a morphism in $\KKK$ and $x\co U(k)(X_0) \to X_1$
is a morphism in $U(K_1)$.
Composition of $(K_0,X_0) \smash{\raisebox{-2pt}{$\xto{(k_0,x_0)}$}} (K_1,X_1)$ and
$(K_1,X_1) \smash{\raisebox{-2pt}{$\xto{(k_1,x_1)}$}} (K_2,X_2)$ is given by 
$$
(k_1,x_1)\circ(x_0,k_0) = (k_1\circ k_0, x_1\circ U(k_1)(x_0)).
$$
This category is also called the Grothendieck construction of $U$ and the notation
$\Tr_\KKK U$ is also used.
Thomason \cite{Thomason} shows that there is a natural weak homotopy
equivalence
\begin{equation}
\label{Thomason map}
\eta\co \hhocolim{\KKK}\, |U| \xto{\ \ \simeq \ \ } |\Tr_{\KKK}U|
\end{equation}
A natural transformation $U\Rightarrow U'$ gives
rise to a canonical functor $\Tr_\KKK U \to \Tr_{\KKK}U'$.
The induced map $|\Tr_\bfK(U)|\to|\Tr_\bfK(U')|$ corresponds via
$\eta$ \eqref{Thomason map} to the
induced map $\hhocolim{\bfK}\, |U| \to \hhocolim{\bfK}\, |U'|$.
Furthermore, for every object $K$ in $\KKK$ the natural map
$$
|U(K)| \to \hhocolim{\bfK}\, |U|
$$
corresponds under \eqref{Thomason map} to the inclusion of categories
\begin{equation}
\label{translation cone}
U(K) \to \Tr_{\KKK}\, U, \qquad \text{ where }  X \mapsto (K,X) \text{ and }
x \mapsto (1_K,x).
\end{equation}
%
%
%
Consider now a functor $F\co\KKK\to\LLL$ of small categories.
Given $U\co \LLL\to\cat$ there is a naturally defined functor
\begin{equation}
\label{def F shriek}
F_! \co  \Tr_\KKK F^*U \to \Tr_\LLL U,
\qquad \text{where} \qquad
\left\{
\begin{array}{l}
F_!(K,X\in F^*U(K)) = (FK,X) \\
F_!(k,x) = (Fk,x).
\end{array}
\right.
\end{equation}
The functor $F_!$ is a model for the map $\hocolim\,F^*|U| \to \hocolim\,|U|$ in the sense
that the following square commutes
$$
\begin{CD}
|\Tr_\KKK F^*U| 
@>{\eta}>>
\hhocolim{\KKK} F^*|U|
\\
@V{|F_!|}VV
@VVV
\\
|\Tr_\LLL U| 
@>{\eta}>>
\hhocolim{\LLL} |U|
\end{CD}
$$

\begin{defn}
\label{def catF star}
For a functor $U\co \KKK\to\cat$ define $F_*U\co \LLL\to\cat$ by
$$
F_*U(L) = \Tr \bigl(
(F\darrow L) \xto{\text{proj}} \KKK \xto{U} \spaces
\bigr)
$$
\end{defn}
The maps $\eta$ \eqref{Thomason map} provide a natural weak homotopy equivalence
$|F_*U| \xto{\simeq} F_*|U|$.
The equivalence in the pushdown theorem can be realized as the nerve of a functor
between the transporter categories as follows.

\begin{prop}
\label{def F sharp}
The functor $F_\# \co  \Tr_\LLL F_*U \to \Tr_\KKK U$ defined by
$$
\begin{array}{l}
F_\# \co  \big( L,(K,FK \to L),X\in UK \big) \mapsto (K,X) \\
F_\# \co  \big( L \xto{\ell} L', K \xto{k} K', U(k)(X) \xto{x} X' \big) \mapsto (k,x).
\end{array}
$$
renders the following diagram commutative where the arrow at the top of the square is 
an equivalence by the pushdown theorem.
\begin{equation}
\xymatrix{
{\hhocolim{\LLL}}\, |F_*U| 
\ar[dr]^\eta_\simeq
&
\hhocolim{\LLL}\, F_*|U|
\ar[r]^{\hbox{\footnotesize$\sim$}}
\ar[d]_{\hbox{\footnotesize$\sim$}}
\ar[l]_{\hbox{\footnotesize$\sim$}}
&
\hhocolim{\KKK}\, |U|
\ar[d]_{\hbox{\footnotesize$\sim$}}^\eta
\\
&
|\Tr_{\LLL}(F_*U)|
\ar[r]_{|F_\#|} &
|\Tr_{\KKK}(U)|
}
\end{equation}
\end{prop}

It is useful to point out that if $\star\co \KKK \to \cat$ is the constant functor on
the trivial category with one object and an identity morphism, then $\Tr_\KKK(\star)=\Nr(\KKK)$.

%
%
\section{EI categories}
\label{sec EI}

Fix an EI category $\A$, namely a category all of whose endomorphisms are isomorphisms.
We shall assume that the category $\A$ is finite.
We shall also assume that $\A$ is equipped with a height function, namely a function
$h\co \Obj(\A) \to \NN$ such that $h(A) \leq h(A')$ if there exists a morphism $A\to A'$ in $\A$ and equality holds if and only if $A\to A'$ is an isomorphism.
Clearly, if $\A$ is an EI-category then so is $\A^\op$.
The finiteness condition also implies that if $\A$ is heighted then so is $\A^\op$.

We can always choose a full subcategory $\A_{\sk}$ of $\A$ which contains one representative from each isomorphism class of objects in $\A$.
We say that $\A_{\sk}$ is skeletal in $\A$.
Clearly the inclusion $\A_{\sk}\subseteq \A$ is an equivalence of categories.
In the language of S\l omi\'nska \cite{Slominska} $\A_{\sk}$ is an EIA category.

Throughout we let $\underk$ denote the poset $\{0\to 1\to\cdots \to k\}$ considered as a 
small category.

\begin{defn}
\label{def sA}
The subdivision category $s(\A)$ is the category whose objects are height increasing functors 
$\AAA\co \underk\to \A$, namely $h(\AAA(i))<h(\AAA(i+1))$ for all $i<k$.
Morphisms $\AAA \to \AAA'$ in $s(\A)$ are pairs $(\epsilon,\vp)$ where 
$\epsilon\co \underk'\to\underk$ is a strictly increasing function and 
$\vp\co \epsilon^*(\AAA) \to \AAA'$ is a natural isomorphism of functors $\underk' \to \A$.
Composition of $(\epsilon,\vp)\co \AAA\to \AAA'$ and $(\epsilon',\vp')\co \AAA' \to \AAA''$ 
is given by $(\epsilon\circ\epsilon',{\epsilon'}^*(\vp)\circ \vp')$.
\end{defn}
Note that $\epsilon$ is determined by the heights of the values of $\AAA$ namely 
$\epsilon(i)=j$ if and only if $h(\AAA'(i))=h(\AAA(j))$.

We shall further assume that $\A$ contains a subcategory $\I$ which is a 
poset with the property that every morphism $\vp\co A \to A'$ in $A$ can be factored uniquely 
as $\vp = \iota \vp'$ where $\vp'$ is an isomorphism in $\A$ and $\iota$ is a morphism in $\I$.
The ladder
$$
\xymatrix{
{\cdots} \ar[r] &
\AAA(n-2) \ar[r]^{\vp_{n-1}} \ar[d]_{(\vp_n'\circ\vp_{n-1})'}^\cong &
\AAA(n-1) \ar[r]^{\vp_n} \ar[d]_{\vp_n'}^\cong &
\AAA(n) \ar@{=}[d] \\
{\cdots} \ar[r] &
\AAA'(n-2) \ar[r]_{\iota} &
\AAA'(n-1) \ar[r]_\iota &
\AAA'(n)
}
$$
shows that the the full subcategory $s_\I(\A)$ of $s(\A)$ consisting
of the objects $\AAA$ in which all the arrows $\AAA(i) \to \AAA(i+1)$
belong to $\I$ is a skeletal subcategory of $s(\A)$. 
We obtain two skeletal subcategories of $s(\A)$
\begin{equation}
\label{skel sA}
s(\A_{\sk}) \subseteq s(\A) \qquad \text{ and } \qquad s_\I(\A)
\subseteq s(\A). 
\end{equation}
We observe that $\Hom_{s(\A)}(\AAA,\AAA')$ has a free action of $\Aut_{s(\A)}(\AAA')$ 
with a single orbit.
Also every $(\epsilon,\vp)\co  \AAA\to \AAA'$ in $s(\A)$ gives rise to a natural group homomorphism upon 
restriction and conjugation with the isomorphism $\vp\co \epsilon^*\AAA\approx\AAA'$
\begin{equation}
\label{sA auto maps}
\vp_*\co  \Aut_{s(\A)}(\AAA) \to \Aut_{s(\A)}(\AAA').
\end{equation}

\begin{prop}
\label{p cofinal}
There is a right cofinal functor $p\co  s(\A) \to \A$ defined by
$$
p(\AAA) = \AAA(0), \qquad \qquad (\AAA\co \underk \to \A).
$$
\end{prop}

\begin{proof}
S{\l}omi\'nska \cite[Proposition 1.5]{Slominska} shows that the functor
$p\co s(\A_{\sk}) \to \A_{\sk}$
is right cofinal (\fullref{def right cofinal}) hence so is $p\co s(\A) \to \A$.
\end{proof}
\begin{defn}
\label{def barsA}
The category $\bar{s}(\A)$ has
the isomorphism classes $[\AAA]$ of the objects of $s(\A)$ as its object set. 
There is a unique morphism $[\AAA] \to [\AAA']$ if there
exists a morphism $\AAA \to \AAA'$ in $s(\A)$.
There is an obvious projection functor
$$
\pi\co  s(\A) \to \bar{s}(\A), \qquad \AAA \mapsto [\AAA].
$$
When $\D$ is a full subcategory of $s(\A)$ one obtains a sub-poset $\bar{\D}$ of $\bar{s}(\A)$
whose objects are the isomorphism classes of the objects of $\D$.
\end{defn}
Clearly $\bar{s}(\A)$ is a poset and it should be compared with S\l omi\'nska's construction
of $s_0(\A)$ in \cite[Section~1]{Slominska}.
Also note that $\bar{s}_\I(\A)=\bar{s}(\A)$ because $s_\I(\A)$ is skeletal in $s(\A)$.
Similarly $\bar{s}(\A_{\sk})=\bar{s}(\A)$.

\begin{lem}
\label{adj cof lem}
Let $J\co {\mathcal{C}} \to \D$ be a functor of small categories with a
left adjoint $L\co \D \to {\mathcal{C}}$ such that 
$L \circ J=\Id$.
Then $J$ is right cofinal.
\end{lem}

\begin{proof}
Fix an object $d\in \D$.
We have to prove that the category $(d \darrow J)$ has a contractible nerve.
Let $(d \darrow \D)$ denote the category $(d \darrow 1_\D)$.
It clearly has a contractible nerve because it has an initial object.
The functors $J$ and $L$ induce obvious functors
\begin{eqnarray*}
&& J_* \co  (d\darrow J) \to (d\darrow \D), \qquad (c, d\xto{f} Jc) \mapsto (Jc, d \xto{f} Jc) \\
&& L_*\co (d\darrow\D) \to (d\darrow J), \qquad (d', d \xto{f} d') 
\mapsto (Ld', d \xto{f} d' \xto{\eta} JLd').
\end{eqnarray*}
It is obvious that $L_*\circ J_*=\Id$.
Furthermore the unit $\eta\co \Id \to LJ$ gives rise to a natural transformation 
$\Id \to J_*\circ L_*$.
Therefore $J_*$ induces a homotopy equivalence on nerves so 
$|(d\darrow J)| \simeq |d\darrow \D)|\simeq *$.
\end{proof}

\begin{prop}
\label{commapi}
For every functor $F\co s(\A) \to\cat$ and every $\AAA\in s(\A)$ there is a functor
$$
\Tr_{\B\Aut(\AAA)} F(\AAA) \to (\pi_* F)([\AAA])
$$
which induces a weak homotopy equivalence on nerves.
It is natural in the sense that every morphism $(\epsilon,\vp)\co \AAA \to \AAA'$ in $s(\A)$, 
gives rise to a square
$$
\xymatrix{
{\Tr_{\Aut(\AAA)} F(\AAA)}
\ar[r] \ar[d]_{\Tr_{\vp_*} F(\vp)} &
(\pi_*F)([\AAA]) \ar[d]^{\pi_*([\vp])}
\ar@{}|-{\overset{\tau}{\Leftarrow}}[dl]
\\
{\Tr_{\Aut(\AAA')} F(\AAA')} \ar[r] &
(\pi_*F)([\AAA']) 
}
$$
which commutes up to a natural transformation $\tau$ which is functorial in $(\epsilon,\vp)$.
Here $\vp_*\co \Aut(\AAA) \to \Aut(\AAA')$ is the homomorphism induced by restriction and
conjugation by $\vp\co \epsilon^*\AAA \approx \AAA'$ as we described in \eqref{sA auto maps}.
The square commutes on the nose if $F(\vp)\co F(\AAA) \longrightarrow F(\AAA')$ is the identity.
\end{prop}

\begin{proof}
Fix an object $\AAA\co \underk \to \A$ in $s(\A)$.
Let $\Pi_\AAA$ be the full subcategory of $(\pi\darrow [\AAA])$
consisting of the objects $(\AAA',[\AAA']\xto{=} [\AAA])$.
It is isomorphic to the full subcategory of $s(\A)$ consisting of the
isomorphism class of $\AAA$.
The inclusion $J\co \Pi_\AAA \to (\pi\darrow [\AAA])$ has a left adjoint
$L$ where
$$
L \co  (\BBB,[\BBB] \to [\AAA]) \mapsto \epsilon^*\BBB, \qquad \text{
  where } [\epsilon^*\BBB]=[\AAA] \text{ for } \epsilon\co \underk' \hookrightarrow \underk.
$$
Clearly $\epsilon$ is unique if it exists.
There is a natural map $\BBB \to \epsilon^*\BBB$ induced
by the identity on $\epsilon^*\BBB$ under the bijection
$s(\A)(\BBB,\epsilon^*\BBB) \approx
s(\A)(\epsilon^*\BBB,\epsilon^*\BBB)$.
We obtain a natural transformation $\Id \to JL$ which gives rise to
bijections for every object $(\AAA',[\AAA']\to [\AAA])$ in $(\pi \darrow [\AAA])$ 
$$\Hom_{(\pi \darrow [\AAA])}(\BBB,J\AAA') = \Hom_{s(\A)}(\BBB,\AAA')
\approx \Hom_{s(\A)}(\epsilon^*\BBB,\AAA') =
\Hom_{\Pi_\AAA}(L\BBB,\AAA').$$
Thus $L$ is left adjoint to $J$ and we apply \fullref{adj cof lem}.
By definition $\Pi_\AAA$ is a connected groupoid with automorphism group 
$\B\Aut_{s(\A)}(\AAA)$.
Therefore upon realization, the functor
$$
\Tr\bigl( \B\Aut_{s(\A)}(\AAA){\to}s(\A) \xto{F} \Cat \bigr) \xto{\text{restriction}}
\Tr\bigl((\pi{\darrow}[\AAA]){\to}s(\A) \xto{F} \Cat\bigr) = (\pi_*F)([\AAA])
$$
induces a weak homotopy equivalence.
Also, for a morphism $\vp\co \AAA\to \AAA'$ we get an obvious 
$\vp_*\co  \Pi_\AAA \to \Pi_\AAA'$ by restriction and conjugation by the isomorphism
$\vp\co \epsilon^*\AAA \to \AAA'$.
It gives rise to the following diagram
$$
\xymatrix{
{\B\Aut_{s(\A)}(\AAA)} \ar[r]^{\incl} \ar[d]^{\vp_*} & 
\Pi_\AAA \ar[d]_{\vp_*} \ar[r]^{J} &
(\pi\darrow [\AAA]) \ar[d]^{[\vp]_*} 
\\
{\B\Aut_{s(\A)}(\AAA')} \ar[r]^{\incl} &
\Pi_{\AAA'}
\ar[r]^J &
(\pi\darrow [\AAA'])
}
$$ 
The morphism $\vp$ provides a canonical natural transformation 
$[\vp]_*\circ J \circ \incl \to J \circ \incl \circ \vp_*$.
This provides the natural transformation $\tau$ in the statement of the proposition and its 
naturality with $\vp$.
If $F(\vp)\co F(\AAA) \longrightarrow F(\AAA')$ is the identity, then $F(\tau)$ becomes the identity and
the square in the statement of the proposition commutes.
\end{proof}

%
%
\section{Proof of the main results}

Fix a $p$--local finite group $(S,\F,\L)$ and an $\F$--centric collection ${\mathcal{C}}$. 
Choose a subcategory $\I \subseteq \L^{\mathcal{C}}$ of distinguished
inclusions, cf \fullref{iota morphisms}.
Note that $\L^{\mathcal{C}}$ possesses a height function, see \fullref{sec EI}, by assigning to a 
subgroup $P$ in ${\mathcal{C}}$ its order.
Also every morphism in $\L^{\mathcal{C}}$ factors uniquely as an isomorphism followed by a morphism in 
$\I$.

We claim that (Definitions \ref{def bsdc} and \ref{def barsA})
$$
\bsd{\mathcal{C}} = \bar{s}(\L^{\mathcal{C}}).
$$
To see this recall that $s_\I(\L^{\mathcal{C}})$ is a skeletal subcategory of $s(\L^{\mathcal{C}})$, see 
\eqref{skel sA},  hence $\bar{s}(\L^{\mathcal{C}}) = \bar{s}_\I(\L^{\mathcal{C}})$.
The functor $\pi\co \L \to \F$ gives a functor $\bar{s}_\I(\L^{\mathcal{C}}) \to \bsd{\mathcal{C}}$ because
it maps the morphisms $\iota\in \I$ to inclusion of subgroups of $S$.
It is an isomorphism of categories because conjugation \eqref{def k-simplices} of two 
$k$--simplices $P_0<\cdots<P_k$ and $P'_0<\cdots< P'_k$ induced by an isomorphism
$\vp_k\in\Iso_\F(P_k,P_k')$ can be lifted to an isomorphism of the corresponding objects
in $s_\I(\L^{\mathcal{C}})$ by lifting the isomorphism $\vp_k\co P_k \to P_k'$ to $\L$ and using 
\fullref{restn in L} to obtain the commutative ladder in $\L^{\mathcal{C}}$:
$$
\xymatrix{
P_0 \ar[r]^{\iota} \ar[d]_\cong &
P_1 \ar[r]^\iota \ar[d]_\cong & 
\cdots \ar[r]^{\iota} &
P_k \ar[d]^{\tilde{\vp_k}}
\\
P_0' \ar[r]_\iota & 
P_1' \ar[r]^\iota &
\cdots \ar[r]^{\iota} &
P_k'
}
$$
We remark that a $k$--simplex $P_0<\cdots<P_k$ in ${\mathcal{C}}$ can be identified with the object
$P_0 \xto{\iota} \cdots \xto{\iota} P_k$ of $s(\L^{\mathcal{C}})$.
Under this identification we clearly have
$$
\Aut_\L(\PPP) = \Aut_{s(\L^{\mathcal{C}})}(\PPP).
$$ 
When $G$ is a discrete group we let 
$\B G$ denote the category with one object and $G$ as its set of morphisms.
For every $k$--simplex $\PPP$ in ${\mathcal{C}}$ 
we identify $\B\Aut_\L(\PPP)$ with the obvious subcategory of $\B\Aut_\L(P_0)$.

\begin{thm}
\label{thmC}
Let ${\mathcal{C}}$ be an $\F$--centric collection in a $p$--local finite group $(S,\F,\L)$.
Then there exists a functor
$\tilde{\de}_{\mathcal{C}}\co \bsd{\mathcal{C}} \to \Cat$
with the following properties
\begin{enumerate}
\item
There is a naturally defined functor
$\Tr_{\bsd{\mathcal{C}}}(\tilde{\de}_{\mathcal{C}}) \to \L^{\mathcal{C}}$
which induces a weak homotopy equivalence on nerves.
\label{thmC:target}

\item
For every $k$--simplex $\PPP$ there is a canonical functor
$
\B\Aut_\L(\PPP) \to \tilde{\de}_{\mathcal{C}}([\PPP])
$
which induces a weak homotopy equivalence on nerves.
If $\PPP'$ is a subsimplex of $\PPP$ then the following square commutes
$$
\xymatrix{
{\B\Aut_\L(\PPP)} \ar[r] \ar[d]_{\res^\PPP_{\PPP'}} &
\tilde{\de}_{\mathcal{C}}([\PPP]) \ar[d] \\
{\B\Aut_\L(\PPP')} \ar[r]  &
\tilde{\de}_{\mathcal{C}}([\PPP'])  
}
$$
\label{thmC:BPtype}

\item
The natural inclusion $\B\Aut_\L(\PPP) \subseteq \L^{\mathcal{C}}$ is equal to the composition
$$
\B\Aut_\L(\PPP) \to \tilde{\de}_{\mathcal{C}}([\PPP]) \subseteq
\Tr_{\bsd{\mathcal{C}}}(\tilde{\de}_{\mathcal{C}}) \to \L^{\mathcal{C}} 
$$
\label{thmC:augment}

\item
An isomorphism of $k$--simplices $\psi\co \PPP'\xto{\approx} \PPP$  in $s(\L^{\mathcal{C}})$ induces a
commutative square
$$
\xymatrix{
{\B\Aut_\L(\PPP)} \ar[r] \ar[d]_{c_\psi} &
\tilde{\de}_{\mathcal{C}}([\PPP]) \ar[d] \\
{\B\Aut_\L(\PPP')} \ar[r]  &
\tilde{\de}_{\mathcal{C}}([\PPP'])  
}
$$
\label{thmC:morphisms}
\end{enumerate}
\end{thm}

\begin{proof}
We have seen that $\bsd{\mathcal{C}} = \bar{s}(\L^{\mathcal{C}})$.
Let $\star\co \bar{s}(\L^{\mathcal{C}}) \to \cat$ denote the constant functor on the trivial
small category with one object and identity morphism.
Use the projection functor 
$\pi\co s(\L^{\mathcal{C}}) \to \bar{s}(\L^{\mathcal{C}})$ to define
$$
\tilde{\de}_{\mathcal{C}} = \pi_*(\star)
$$
According to \fullref{commapi} we have a canonical functor
$$
\B\Aut_\L(\PPP) = 
\Tr_{\B\Aut_\L(\PPP)}(\star) \to \pi_*(\star)([\PPP]) = \tilde{\de}_{\mathcal{C}}([\PPP])
$$
which induces a weak homotopy equivalence.
Since $\star$ is constant, the square in the statement of \fullref{commapi}
commutes and we obtain the naturality assertions in point \eqref{thmC:BPtype} and
\eqref{thmC:morphisms}.
The natural functor of \eqref{thmC:target} is defined using \fullref{def F sharp} and
\fullref{p cofinal} by
$$
\Tr_{\bsd{\mathcal{C}}} (\tilde{\de}_{\mathcal{C}}) =
\Tr_{\bar{s}(\L^{\mathcal{C}})} (\pi_*(\star)) 
\xto{ \ \pi_\# \ } 
\Tr_{s(\L^{\mathcal{C}})}(\star) = s(\L^{\mathcal{C}}) \xto{p} \L^{\mathcal{C}}.
$$
It induces a weak homotopy equivalence by \fullref{pushdown thm},
\fullref{def F sharp} and \fullref{cofinal thm}.
Whence point \eqref{thmC:target}.
Inspection of the functor $\pi_\#$, the inclusion
$\B\Aut_\L(\PPP) \subseteq (\pi\darrow [\PPP]) = \pi_*(\star)([\PPP])$ and
\fullref{translation cone} yield point \eqref{thmC:augment}
\end{proof}

\begin{proof}[Proof of \fullref{thmA}]
Apply  \fullref{thmC} above and define $\de_{\mathcal{C}} = |\tilde{\de}_{\mathcal{C}}|$.
\end{proof}

\begin{void}
\label{compare with Dwyer}
We now relate the construction in \fullref{thmA} to Dwyer's normaliser decomposition
\cite[Section~3]{Dwyer-decomp}.
We will show that the two functors are related by a zigzag of natural transformations which
induce a mod--$p$ equivalence. 

Fix a finite group $G$ and the $p$--local finite group $(S,\F,\L)$ associated with it.
A collection ${\mathcal{C}}$ of $\F$--centric subgroups of $S$ gives rise
to a $G$--collection
$\H$ of $p$--centric subgroups of $G$ (cf \cite[Section~1.19]{Dwyer-decomp}, 
\fullref{non exotic examples}) 
by taking all the $G$--conjugates of the elements of ${\mathcal{C}}$.
We let $\T^\H$ denote the transporter category of $\H$.
That is, the object set of $\T^\H$ is $\H$ and the morphism set $\T^\H(H,K)$ is the
set $N_G(H,K) = \{ g\in G : g^{-1}Hg\leq K\}$.
We also let $\T^{\mathcal{C}}$ denote the full subcategory of $\T^\H$ having ${\mathcal{C}}$ as its object set.
Almost by definition $\T^{\mathcal{C}}$ is skeletal in $\T^\H$.
We also obtain a zigzag of functors (see \fullref{non exotic examples})
$$
\T^\H \leftarrow \T^{\mathcal{C}} \to \L^{\mathcal{C}}.
$$
Dwyer \cite[Section~3]{Dwyer-decomp} defines a category $\bsd\H$ whose objects are the 
$G$--conjugacy classes $[\HHH]$ of the $k$--simplices $H_0<\cdots<H_k$ in $\H$.
There is a unique morphism $[\HHH] \to [\HHH']$ in $\bsd\H$ if and only if $\HHH'$ is
conjugate in $G$ to a subsimplex of $\HHH$.
It follows directly from the definition of $\H$ as the smallest $G$--collection containing
${\mathcal{C}}$ and from the definition of $\F=\F_S(G)$ that $\bsd\H=\bsd{\mathcal{C}}$.
We obtain a commutative diagram (see \fullref{def sA})
$$
\xymatrix{
s(\T^\H) \ar[d]_{\pi_2} &
s(\T^{\mathcal{C}}) \ar[r] \ar[d]_{\pi_1} \ar[l]_{\supseteq} &
s(\L^{\mathcal{C}}) \ar[d]^\pi 
\\
{\bar{s}}(\T^\H) &
\bar{s}(\T^{\mathcal{C}}) \ar@{=}[l] \ar@{=}[r] &
\bar{s}(\L^{\mathcal{C}}).
}
$$
Fix a $k$--simplex $\PPP=P_0<\cdots<P_k$ in ${\mathcal{C}}$.
Note that $(\pi_2\darrow [\PPP])$ 
is isomorphic to the subcategory of $\T^\H$ of the objects $\PPP'$ which admit a morphism
to $\PPP$.
It contains a full subcategory $\Pi_\PPP$ of the objects
of $s(\T^\H)$ that are isomorphic to $\PPP$; cf the proof of \fullref{commapi}.
By inspection $\Pi_\PPP$ is the translation category of the action of $G$ on the orbit of
$\PPP$, that is it is the transported category of the $G$--set $[\PPP]$ in $\H$ thought of 
as a functor $\B\Aut_\L(\PPP)\to \sets$, cf \cite[Section~3.3]{Dwyer-decomp}.
Thus
$$
\de^{\text{Dwyer}}_\H([\PPP]) = EG \times_G [\PPP] = \Nr(\Pi_\PPP).
$$
The inclusion $J\co \Pi_\PPP\hookrightarrow (\pi_2 \darrow [\PPP])$ has a left adjoint 
$L\co (\PPP', [\PPP'] \to [\PPP])  \mapsto \epsilon^*\PPP'$ 
where $(\epsilon,\vp)\co \PPP' \to \PPP$ is a morphism in $s(\L^{{\mathcal{C}}})$, see \fullref{def sA}.
Compare the proof of \fullref{commapi}.
\fullref{adj cof lem} implies that $J$ is right cofinal.
We obtain a zigzag of functors
$$
\de^{\text{Dwyer}}_\H \xrightarrow[\simeq]{\phantom{-----}} 
|(\pi_2)_*(\star)|
\xleftarrow[\simeq]{\ \incl \ } |(\pi_1)_*(\star)| \xto{\text{ mod-}p \ }
|\pi_*(\star)| = |\tilde{\de}_{\mathcal{C}}| = \de_{\mathcal{C}}.
$$
The third map induces a mod--$p$ equivalence by the following argument.
For any object $[\PPP]$ we obtain a map
$(\pi_1)_*(\star)([\PPP]) \to \pi_*(\star)([\PPP])$
which by \fullref{commapi} is equivalent to the map 
$$
B\Aut_G(\PPP) \to B\Aut_\L(\PPP).
$$
Since $P_0$ is $p$--centric then $C_G(P_0)=Z(P_0)\times C_G'(P_0)$ where $C_G'(P_0)$ is
a characteristic $p'$--subgroup of $C_G(P_0)$ and $\Aut_\L(P_0) = N_G(P_0)/C_G'(P_0)$.
Therefore $\Aut_G(\PPP) \to \Aut_G(\PPP)/C_G'(P_0)$ induces a mod--$p$ equivalence as needed.
\end{void}


We shall now prove \fullref{thmB}.
Fix an $\F$--centric collection ${\mathcal{C}}$ and a collection $\E$ of elementary abelian
subgroups in $(S,\F,\L)$.
Recall from \fullref{def barcl} and \fullref{def NoL}
the definitions of $\bar{C}_\L({\mathcal{C}};E_k)$ and $\No_\L({\mathcal{C}};\EEE)$ 
where $\EEE$ is a $k$--simplex in $\E$.

\begin{prop}
\label{NoTrCbar}
Fix a $k$--simplex $\EEE$ in $\E$, namely a functor $\EEE\co \underk \to \F^\E$.
There is a functor 
$$
\epsilon\co  \No_\L({\mathcal{C}};\EEE) \to 
    \Tr_{\B\Aut_\F(\EEE)} \Big(\bar{C}_\L({\mathcal{C}};\EEE) \Big)
$$
which is fully faithful and natural with respect to inclusion of simplices.
If $E_k$ is fully $\F$--centralised, its image is also skeletal and
in particular induces homotopy equivalence
$
|\No_\L({\mathcal{C}};\EEE)| \xto{ \ \ \simeq \ \ } 
    |\bar{C}_\L({\mathcal{C}};E_k)|_{h\Aut_\F(\EEE)}.
$
\end{prop}

\begin{proof}
The objects of $\H:=\B\Aut_\F(\EEE)\int \bar{C}_\L({\mathcal{C}};\EEE)$ are pairs
$(P,f)$ where $P\in {\mathcal{C}}$ and $f \in \F(E,Z(P))$.
Morphisms are pairs $(\vp,g)$ where $\vp\in \L(P,P')$ and
$g\in\Aut_\F(\EEE)$ such that $f'=\pi(\vp)\circ f \circ g$ (see \fullref{sec homotopy colimits}).
Since $f,f'$ are monomorphisms, $g$ is determined by $\vp$. 
Define $\epsilon\co \No_\L({\mathcal{C}};\EEE) \to \H$ by
\begin{eqnarray}
 &&\epsilon(P) = (P,E \xto{\incl} Z(P) \leq P) \\
\nonumber
 &&\epsilon( P\xto{\vp} P') = (P\xto{\vp} P', \pi(\vp)|_{E_k}^{-1}).
\end{eqnarray}
It is well defined  and fully faithful by the definition of $\No_\L({\mathcal{C}};\EEE)$.
Naturality with respect to inclusion of simplices is readily verified.
Consider an object $(P,f) \in \H$.
Note that $f(E_k)\leq Z(P)$, hence for $g:=f^{-1}\in \Iso_\F(f(E_k),E_k)$ we must have
(see \fullref{def sat fus}) $N_g \supseteq C_S(f(E_k)) \supseteq P$.
By axiom II of \fullref{def sat fus} we can extend $g$ to an isomorphism
$h\co P \to P'$
in $\F$.
Clearly $P'$ is in ${\mathcal{C}}$ because the latter is a collection.
Fix a lift $\tilde{h}\in\L(P,P')$ for $h$.
We have 
$\pi(\tilde{h})\circ f = h \circ \smash{\incl^P_{f(E_k)}} \circ g^{-1} =
\smash{\incl_{E_k}^{P'}}$.
Therefore $(\id_{E_k},\tilde{h})$ is an isomorphism $(P,f) \cong
\bigl(P',\smash{\incl_{E_k}^{P'}}\bigr)$
in $\H$.
This shows that $\epsilon$ embeds $\No_\L({\mathcal{C}};\EEE)$ into a skeletal subcategory of $\H$ and
the result follows.
\end{proof}

\begin{proof}[Proof of \fullref{thmB}]
For every $P\in{\mathcal{C}}$ define $\zeta(P)=\Omega_pZ(P)$, see \fullref{def omega_p}.
Note that if $f\co P\to P'$ is an isomorphism in $\F$ then $f^{-1}\co \zeta(P') \to \zeta(P)$
is an isomorphism in $\F$.
Also if $P \leq P'$ in ${\mathcal{C}}$ then $Z(P) \leq Z(P')$ because $P$
and $P'$ are $\F$--centric so
their centres are equal to their $S$--centralisers.
It easily follows that this assignment forms a functor
$$
\zeta\co  \L^{\mathcal{C}} \to {\F^\E}^\op.
$$ 
Fix $E\in\E$.
Since every homomorphism $f\co E \to Z(P)$ factors through $\zeta(P)$ we see that
$\bar{C}_\L({\mathcal{C}};E) =(\zeta \darrow E)$.
In particular  \eqref{def catF star} 
\begin{equation}
\label{kan zeta}
\zeta_*(\star) \co  E \mapsto \bar{C}_\L({\mathcal{C}};E).
\end{equation}
We now observe that $\F^\E$ is an EI-category.
The assignment $E \mapsto |E|$ gives rise to a height function in the sense of
\fullref{sec EI}.
Furthermore the set $\I$ of inclusions in $\F^\E$ forms a poset where every morphism
in $\F^\E$ factors uniquely as an isomorphism followed by an element in $\I$.
We conclude that
$$
\bar{s}\E = \bar{s}_\I(\F^\E) \approx \bar{s}(\F^\E).
$$
There is an isomorphism $\tau\co s({\F^\E}^\op) \to s(\F^\E)$ which is the identity on objects.
On the morphism set between $\EEE\co \underk\to\F^\E$ and $\EEE'\co \undern\to\F^\E$ such that
$\epsilon^*\EEE\approx \EEE'$ for some injective $\epsilon\co \undern \to \underk$ (see 
\fullref{def sA}), $\tau$ has the effect
\begin{multline*}
\Hom_{s({\F^\E}^\op)}(\EEE,\EEE') =
\Iso_{{\F^\E}^\op}(\epsilon^*\EEE,\EEE') = 
\\
\Iso_{\F^\E}(\EEE',\epsilon^*\EEE)
\xrightarrow[\approx]{\ \ \ f \mapsto f^{-1} \ \  } 
\Iso_{\F^\E}(\epsilon^*\EEE,\EEE') = \Hom_{s(\F^\E)}(\EEE,\EEE').
\end{multline*}
The functor $p\co s({\F^\E}^\op) \to {\F^\E}^\op$ of \fullref{p cofinal} fits into a commutative
diagram
$$
\xymatrix{
s({\F^\E}^\op) 
\ar[r]^p \ar[d]_\tau^\approx &
{\F^\E}^\op 
\ar@{=}[d] 
\\
s(\F^\E) 
\ar[r]^\mu &
{\F^\E}^\op
}
$$
where $\mu(\EEE)=E_k$.
Since $\tau$ is an isomorphism, $\mu$ is right cofinal.
We obtain a zigzag of functors
$$
\bsd\E \xleftarrow{\ \ \pi \ \ } 
s(\F^\E) \xto{\ \ \mu \ \ }
{\F^\E}^\op
\xleftarrow{ \ \ \zeta \ \ }
\L^{\mathcal{C}} 
\xto{\ \ \star \ \ }
\Cat.
$$
Define 
$$
\tilde{\de}_\E= \pi_*\circ \mu^* \circ \zeta_*(\star), \qquad
\text{ and } \qquad
\de_\E=|\tilde{\de}_\E|.
$$
Since $\mu$ is right cofinal, the \fullref{cofinal thm} and 
\fullref{pushdown thm} imply a weak homotopy equivalence
$$
\hhocolim{\bsd\E} \, \de_\E \xto{\simeq} |\L^{\mathcal{C}}|.
$$
This is point \eqref{thmB:target} of the theorem.
Inspection of \eqref{translation cone}, \fullref{def F sharp} and \eqref{def F shriek} 
shows that this equivalence is given as the realization of a functor 
$\Tr_{\bsd\E}\tilde{\de}_\E \to \L^{\mathcal{C}}$ where
\begin{eqnarray*}
 (\EEE, E_k \xto{f} P) &\mapsto& P \\
 (\EEE \to \EEE', g\in\Aut_\F(\EEE'), P \xto{\vp\in\L} P') &\mapsto& \vp.
\end{eqnarray*}
Point \eqref{thmB:terms} follows from \eqref{kan zeta} and \fullref{commapi}.
Point \eqref{thmB:augment} follows from \fullref{No equiv orbit}, inspection
of $\epsilon$ in \fullref{NoTrCbar}, of \eqref{kan zeta} and the 
functor $\Tr_{\bsd\E}(\tilde{\de}_\E) \to \L^{\mathcal{C}}$.
Point \eqref{thmB:maps} is a consequence of \fullref{commapi}.
%
%
%
\end{proof}

\bibliographystyle{gtart}
\bibliography{link}

\begin{thebibliography}{}
\providecommand\bibmarginpar{\leavevmode\marginpar}
\def\urlstyle#1{{\tt #1}}

\bibitem{BK}
\textbf{A\,K Bousfield}, \textbf{D\,M Kan}, \emph{Homotopy limits, completions
  and localizations}, Springer, Berlin (1972) \xox{MR}{0365573}

\bibitem{BCGLO1}
\textbf{C Broto}, \textbf{N Castellana}, \textbf{J Grodal}, \textbf{R Levi},
  \textbf{B Oliver}, \emph{Subgroup families controlling $p$--local groups}, J.
  London Math. Soc. to appear

\bibitem{BLO1}
\textbf{C Broto}, \textbf{R Levi}, \textbf{B Oliver}, \emph{Homotopy
  equivalences of $p$--completed classifying spaces of finite groups}, Invent.
  Math. 151 (2003) 611--664 \xox{MR}{1961340}

\bibitem{BLO2}
\textbf{C Broto}, \textbf{R Levi}, \textbf{B Oliver},
  \href{http://dx.doi.org/10.1090/S0894-0347-03-00434-X} {\emph{The homotopy
  theory of fusion systems}}, J. Amer. Math. Soc. 16 (2003) 779--856
  \xox{MR}{1992826}

\bibitem{Dwyer-decomp}
\textbf{W\,G Dwyer}, \href{http://dx.doi.org/10.1016/S0040-9383(96)00031-6}
  {\emph{Homology decompositions for classifying spaces of finite groups}},
  Topology 36 (1997) 783--804 \xox{MR}{1432421}

\bibitem{Dw-Ka}
\textbf{W\,G Dwyer}, \textbf{D\,M Kan},
  \href{http://dx.doi.org/10.1016/0040-9383(84)90035-1} {\emph{A classification
  theorem for diagrams of simplicial sets}}, Topology 23 (1984) 139--155
  \xox{MR}{744846}

\bibitem{Goerss-sht}
\textbf{P\,G Goerss}, \textbf{J\,F Jardine}, \emph{Simplicial homotopy theory},
  Progress in Mathematics 174, Birkh\"auser Verlag, Basel (1999)
  \xox{MR}{1711612}

\bibitem{Ho-Vo}
\textbf{J Hollender}, \textbf{R\,M Vogt},
  \href{http://dx.doi.org/10.1007/BF01190675} {\emph{Modules of topological
  spaces, applications to homotopy limits and $E_{\infty}$ structures}}, Arch.
  Math. $($Basel$)$ 59 (1992) 115--129 \xox{MR}{1170635}

\bibitem{Li}
\textbf{A Libman}, \emph{A Minami--Webb splitting of classifying spaces of
  finite groups and the exotic examples of Ruiz and Viruel}, preprint

\bibitem{LV}
\textbf{A Libman}, \textbf{A Viruel}, \emph{On the homotopy type of the
  non-completed classifying space of a $p$--local finite group}, preprint

\bibitem{MacLane-working}
\textbf{S Mac~Lane}, \emph{Categories for the working mathematician}, Graduate
  Texts in Mathematics 5, Springer, New York (1998) \xox{MR}{1712872}

\bibitem{May-ss}
\textbf{J\,P May}, \emph{Simplicial objects in algebraic topology}, Van
  Nostrand Mathematical Studies, No. 11, D. Van Nostrand Co.,, Princeton,
  N.J.-Toronto, Ont.-London (1967) \xox{MR}{0222892}

\bibitem{Puig}
\textbf{L Puig}, \emph{Unpublished notes}

\bibitem{Quillen}
\textbf{D Quillen}, \emph{Higher algebraic $K$--theory. {I}}, from: ``Algebraic
  $K$-theory, I: Higher $K$-theories (Proc. Conf., Battelle Memorial Inst.,
  Seattle, Wash., 1972)'', Springer, Berlin (1973)  85--147. Lecture Notes in
  Math., Vol. 341 \xox{MR}{0338129}

\bibitem{RV-extraspecial}
\textbf{A Ruiz}, \textbf{A Viruel},
  \href{http://dx.doi.org/10.1007/s00209-004-0652-1} {\emph{The classification
  of $p$--local finite groups over the extraspecial group of order $p^3$ and
  exponent $p$}}, Math. Z. 248 (2004) 45--65 \xox{MR}{2092721}

\bibitem{Slominska}
\textbf{J S{\l}omi{\'n}ska}, \emph{Homotopy colimits on {E}-{I}-categories},
  from: ``Algebraic topology Pozna\'n 1989'', Lecture Notes in Math. 1474,
  Springer, Berlin (1991)  273--294 \xox{MR}{1133907}

\bibitem{Thomason}
\textbf{R\,W Thomason}, \emph{Homotopy colimits in the category of small
  categories}, Math. Proc. Cambridge Philos. Soc. 85 (1979) 91--109
  \xox{MR}{510404}

\end{thebibliography}

\end{document}